\documentclass[11pt]{article}
\usepackage{amsmath,amsthm,amssymb}
\usepackage{mathtools}
\usepackage[hyperfootnotes=false]{hyperref}
\usepackage{color}
\usepackage{cancel}
\usepackage{titlesec}
\titleformat{\paragraph}
{\normalfont\normalsize\bfseries}{\theparagraph}{1em}{}
\titlespacing*{\paragraph}
{0pt}{3.25ex plus 1ex minus .2ex}{1.5ex plus .2ex}

\def\titlerunning#1{\gdef\titrun{#1}}
\makeatletter
\def\author#1{\gdef\autrun{\def\and{\unskip, }#1}\gdef\@author{#1}}
\def\address#1{{\def\and{\\\hspace*{18pt}}\renewcommand{\thefootnote}{}%
\footnote {#1}}%
\markboth{\autrun}{\titrun}}
\makeatother
\def\email#1{\hspace*{4pt}{\em e-mail}: #1}
\def\MSC#1{{\renewcommand{\thefootnote}{}%
\footnote{\emph{Mathematics Subject Classification (2020):} #1}}}
\def\keywords#1{\par\medskip
\noindent\textbf{Keywords:} #1}

\newtheorem{theorem}{Theorem}[section]

\newtheorem{prop}[theorem]{Proposition}

\newtheorem{lemma}[theorem]{Lemma}

\newtheorem{reason}{Reason}

\theoremstyle{definition}

\numberwithin{equation}{section}

\def\cH{\mathcal H}

\def\PG{{\rm PG}}

\def\GL{{\rm GL}}

\def\x{\boldsymbol x}
\def\y{\boldsymbol y}
\def\F{\mathbb F}
\def\Gq{{\cal G}_q(k, \perp)}
\def\Gqthree{{\cal G}_q(3, \perp)}
\def\Tr{\mathrm Tr}

\begin{document}


\baselineskip=16pt

\titlerunning{}

\title{A clique-free pseudorandom subgraph of the pseudo polarity graph}

\author{Sam Mattheus
\and 
Francesco Pavese}

\date{}

\maketitle

\address{S. Mattheus: Department of Mathematics, Vrije Universiteit Brussel, Pleinlaan 2, 1050, Brussel, Belgium; \email{sam.mattheus@vub.ac.be}
\and
F. Pavese: Dipartimento di Meccanica, Matematica e Management, Politecnico di Bari, Via Orabona 4, 70125 Bari, Italy; \email{francesco.pavese@poliba.it}}


\MSC{Primary 05D10. Secondary 05E30; 51E20.}

\begin{abstract}
We provide a new family of $K_k$-free pseudorandom graphs with edge density $\Theta(n^{-1/(k-1)})$, matching a recent construction due to Bishnoi, Ihringer and Pepe \cite{BIP2020}. As in the former result, the idea is to use large subgraphs of polarity graphs, which are defined over a finite field $\F_q$. While their construction required $q$ to be odd, we will give the first construction with $q$ even.
\end{abstract}

\section{Introduction}

As described in the excellent survey by Krivelevich and Sudakov \cite{KS}, pseudorandom graphs are deterministic graphs that behave like random graphs with the same edge density. There are many ways, both combinatorially and algebraically, to characterize pseudorandomness quantitatively and we refer to the aforementioned survey for an overview. Their systematic study was initiated by Thomason in \cite{T87a,T87b} and they have found many applications in graph theory and theoretical computer science ever since. 

Pseudorandom graphs can be defined as follows. Recall that an \emph{$(n,d,\lambda)$-graph} is a $d$-regular graph on $n$ vertices such that every eigenvalue except the largest is bounded in absolute value by $\lambda$. Note that when considering eigenvalues of graphs, we refer to the eigenvalues of its adjacency matrix. Whenever $\lambda$ is much smaller than $d$, an $(n,d,\lambda)$-graph has some strong properties in common with random graphs, see for example \cite[Chapter 9.2]{AS16} for several such results. Moreover, if $d \le (1-\epsilon)n$, for some small positive $\epsilon$, one can find $\lambda = \Omega(\sqrt{d})$ by considering the trace of the square of the adjacency matrix. Therefore, $(n, d, \lambda)$-graphs such that $\lambda = O(\sqrt{d})$ are said to be {\em optimally pseudorandom graphs} and instances of these graphs can be obtained from strongly regular graphs, Cayley graphs, or from the polarity of a finite projective space or of a generalized polygon \cite{KS}.

As one of their most important applications, optimally pseudorandom graphs provide explicit lower bounds for Ramsey numbers. Recall that the (multicolour) Ramsey number $r(k_1, \dots, k_m)$ is the minimum number $n$ of vertices such that every $m$-colouring graph of the complete graph $K_n$ has a monochromatic copy of $K_{k_i}$ for some $i \in \{1,\dots,m\}$. In this case, one requires that the family of pseudorandom graphs has no cliques of a fixed size, while still having the largest edge density $d/n$ possible. It is known (using the expander mixing lemma, see for example \cite[Corollary 9.2.5]{AS16}) that the edge density of an optimally pseudorandom $K_k$-free graph is $O(n^{-1/(2k-3)})$. If optimally pseudorandom graphs with this edge density exist, it follows from the results of Mubayi and Verstra\"ete \cite{MV}, and the multicolour generalization by He and Wigderson \cite{HW20}, 
that the determination of the off-diagonal Ramsey number $r(k_1,\dots,k_m, l)$ for fixed $k_1,\dots,k_m \geq 3$ and $l \rightarrow \infty$ is essentially solved.
Therefore the question of finding optimally pseudorandom $K_k$-free graphs with highest possible edge density is one of the most interesting and challenging open problems in the theory of pseudorandom graphs.

Optimally pseudorandom $K_k$-free graphs with edge density $d/m = \Theta(n^{-1/(2k-3)})$ are only known to exist in the case when $k = 3$, by a construction of Alon \cite{A}. In \cite{AK}, Alon and Krivelevich obtained optimally pseudorandom $K_{k}$-free graphs with $\frac{d}{n} = \Omega \left( n^{-1/(k-2)} \right)$, by considering a symmetric non--alternating bilinear form of $\F_q^{k-1}$, $q$ even. Recently, Bishnoi, Ihringer and Pepe \cite{BIP2020}, by using a symmetric bilinear form of $\F_q^k$, $q$ odd, provided a construction of optimally pseudorandom $K_{k}$-free graphs with $\frac{d}{n} = \Omega \left(n^{-1/(k-1)} \right)$ and this is currently the best known construction in terms of edge density for optimally pseudorandom $K_k$-free graphs, $k \ge 4$. Here we give a new family of optimally pseudorandom $K_{k}$-free graphs with $\frac{d}{n} = \Omega \left( n^{-1/(k-1)} \right)$, arising from a symmetric non-alternating bilinear form of $\F_q^k$, $q$ even. Although the underlying construction is different, this family of graphs extends the result for $k=3$ in \cite{MP2019} to higher dimensions.


\section{Polarity graphs}

%

Before we get to the construction, we first introduce some terminology from finite geometry. Let $\PG(k-1, q)$ be the $(k-1)$--dimensional projective space over the finite field $\F_q$. Points and hyperplanes of $\PG(k-1, q)$ can be described as follows. Define an equivalence relation on the set of all nonzero vectors $(x_1, \dots, x_k)$ in $\F_q^k$, by saying that two vectors are equivalent if one is a multiple of the other by an element of $\F_q$. The points of $\PG(k-1, q)$ as well as the hyperplanes can be represented by the equivalence classes of this equivalence relation. A point $(x_1, \dots, x_k)$ belongs to the hyperplane $(y_1, \dots, y_k)$ if and only if $\sum_{i = 1}^k x_i y_i = 0$. Hence the hyperplane $(a_1, \dots, a_k)$ consists of the set of points satisfying the relation $a_1X_1 + \dots + a_kX_k = 0$. 

A {\em polarity} $\perp$ of $\PG(k-1, q)$ is associated to a non--degenerate reflexive sesquilinear form $\beta$ of $\F_q^k$. In particular, given any representative vector $\x$ of a point in $\PG(k-1, q)$, the polarity sends $\x$ to the hyperplane $\x^\perp = \{\y \in \PG(k-1, q) \colon \beta(\x, \y) = 0\}$, whereas for a projective subspace $S = \langle \x_1, \dots, \x_m \rangle$, we define $S^\perp = \cap_i \x_i^\perp$. In this way, $\perp$ is an involutory bijective map sending points to hyperplanes and vice versa, reversing incidence. A point $\x$ is called {\em absolute} with respect to $\perp$ if $\x \in \x^\perp$. If $\beta$ is alternating or Hermitian, then the polarity $\perp$ is called {\em symplectic} or {\em unitary}; if $q$ is odd and $\beta$ is symmetric, then $\perp$ is said {\em orthogonal}, whereas if $q$ is even and $\beta$ is symmetric non--alternating, then $\perp$ is {\em pseudo-symplectic}. A map $g$ of $\GL(\F_q^k)$ is an {\em isometry of $\beta$} if $\beta(\x^g, \y^g) = \beta (\x, \y)$, for all $\x, \y \in \F_q^k$. Similarly two sesquilinear forms $\beta, \gamma$ of $\F_q^k$ are {\em equivalent} if there is a map $g \in \GL(\F_q^k)$ and $a \in \F_q^\times$ such that $\gamma(\x, \y) = a \beta(\x^g, \y^g)$, for all $\x, \y \in \F_q^k$. To any linear map $g \in \GL(\F_q^k)$ corresponds a projectivity of $\PG(k-1, q)$, i.e. a bijection between projective subspaces of $\PG(k-1, q)$ of given dimension preserving incidences. A projectivity of $\PG(k-1, q)$ associated with an isometry of $\beta$ {\em preserves the polarity $\perp$.} We shall find it helpful to represent linear maps of $\F_q^k$ with members of $\GL(k, q)$ acting on the left.      

The {\em polarity graph} $\Gq$ of $\PG(k-1, q)$ with respect to a polarity $\perp$ is the graph having as vertex set the points of $\PG(k-1, q)$ and two distinct points $\x, \y$ are adjacent if and only if $\x \in \y^\perp$. Polarity graphs provide a good basis for clique-free pseudorandom graphs for three reasons.

\begin{reason}
	The eigenvalues of $\Gq$ are $d = (q^{k-1}-1)/(q-1)$ and $\pm q^{(k-1)/2} = O(\sqrt{d})$. 
\end{reason}

\begin{proof}
	It suffices to see that the square of its adjacency matrix $A$ satisfies
	\begin{align*}
		A^2 = q^{k-2}I +  \frac{q^{k-2} -1}{q-1}J,
	\end{align*}
	where $I$ and $J$ denote the identity matrix and the all-one matrix of appropriate size respectively.
\end{proof}	

\begin{reason}
	Polarity graphs have edge density $d/n = \Theta(1/q) = \Theta(n^{-1/(k-1)})$.
\end{reason}
\begin{proof}
	The number of vertices is $(q^k-1)/(q-1)$ and each vertex is adjacent to $(q^{k-1}-1)/(q-1)$ vertices, giving the result.
\end{proof}

This implies that they are pseudorandom and quite dense, but we have to be careful as they are not simple graphs. This can be fixed by taking a subgraph that avoids the vertices with loops, which correspond to absolute points. A nice property of pseudorandom graphs is that after taking a large subgraph, we again find a pseudorandom graph of the same edge density.

\begin{prop}\label{subgraph}
	Consider a $d_H$-regular induced subgraph $H$ of a pseudorandom $(n,d,\lambda)$-graph $G$ of positive density, i.e. $|V(H)| := n_H = \alpha n$, $\alpha > 0$. Then $H$ is again pseudorandom and equally dense, i.e. $\Theta(d/n) = \Theta(d_H/n_H)$.
\end{prop}

\begin{proof}
	The proof is another application of the expander mixing lemma which we mentioned before and interlacing of eigenvalues, for which \cite{Hae95} is a standard reference.
\end{proof}

Now upon taking the subgraph of non-absolute points, we even find that it is clique-free.

\begin{reason}
	The subgraph of non-absolute points of $\Gq$ is $K_{k+1}$-free.
\end{reason}
\begin{proof}
	Suppose that $\x_1, \dots, \x_{k+1}$ are pairwise orthogonal points in $\PG(k-1, q)$, then we obtain that the corresponding vectors in $\F_q^{k}$ are linearly independent. 
Therefore, we find $k+1$ linearly independent vectors in a $k$-dimensional vector space, which is a contradiction. 
\end{proof}

Note that if $\perp$ is symplectic, then the subgraph of non-absolute points of $\Gq$ is empty, whereas if $k$ and $q$ are odd, and $\perp$ is orthogonal, then the subgraph of non-absolute points of $\Gq$ is not regular. In all the remaining cases it can be checked that the subgraph of non-isotropic points of $\Gq$ is an optimally pseudorandom $K_{k+1}$-free graph with $d/n = \Theta(n^{-1/(k-1)})$. 

These three reasons, along with the construction of the subgraph of non-absolute points are also contained in \cite{AK}. The main insight of \cite{BIP2020} is that it is possible, when $q$ is odd and $\perp$ orthogonal, to take a large subgraph on non-absolute points which reduces the clique number even further. In this way, they find a subgraph of the same density $d/n = \Theta(n^{-1/(k-1)})$, but which is $K_k$-free. We will employ the same idea for even $q$ and pseudo-symplectic $\perp$.

\section{Construction}


Let $q = 2^h$ and denote by $U_i$ the point of $\PG(k-1, q)$ represented by the vector having one in the $i$-th position and zero elsewhere. Consider the pseudo-polarity $\perp$ associated with the bilinear form $\beta$ defined as follows:
\begin{align*}
& x_1 y_2 + x_2 y_1 + \dots + x_{k-2} y_{k-1} + x_{k-1} y_{k-2} + x_{k} y_{k} && \mbox{ if } k \mbox{ is odd, } \\ 
& x_1 y_2 + x_2 y_1 + \dots + x_{k-1} y_{k} + x_{k} y_{k-1} + x_{k} y_{k} && \mbox{ if } k \mbox{ is even. } 
\end{align*}
The set of absolute points of $\perp$ are all the points of the hyperplane $H_\infty: X_{k} = 0$ and $H_\infty^\perp$ equals $U_{k}$ or $U_{k-1}$, according as $k$ is odd or even, respectively.

Let $T = \{a \in \F_q \;\; | \;\; \Tr(a) = 1\}$ be the elements of absolute trace one, recalling $\Tr: x \in \F_q \mapsto x + x^2 + x^4 + \dots + x^{2^{h-1}} \in \F_2$. It is easily seen that $|T| = q/2$. Fix an element $t_0 \in T$ and define the hyperplane $H_t: t_0 X_{k-1} + t X_{k} = 0$, where $t \in T$. If $k \ge 3$, its image $H_t^\perp$ lies on $\ell$, where 
\[\ell = \begin{cases}
	U_{k-2}U_{k} & \mbox{ if } k \mbox{ is odd, } \\
	U_{k-1}U_{k} & \mbox{ if } k \mbox{ is even. }
\end{cases}\]
When $k$ is odd, we observe that the line $\ell$ is contained in the hyperplane $X_{k-1} = 0$. When $k$ is even, $\ell$ intersects $H_t$, $t \in T$, in one point. 

Let $\cH(k,q)$ be the induced subgraph of $\Gq$ on the vertex set $\cup_{t \in T} H_t \setminus \left(H_\infty \cup \ell\right)$. We will prove in a few steps that this graph is indeed a $K_k$-free pseudorandom graph. Since the definition depends on the parity of $k$, our proofs will also be split into these two cases.

\begin{lemma}\label{transitivity0}
There is a group of projectivities of $\PG(k-1, q)$, $k \ge 3$, which preserves $\perp$ and acts transitively on points of $H_t \setminus \left(H_\infty \cup \ell\right)$, $t \in T$.  
\end{lemma}
\begin{proof}
If $k$ is odd, the linear maps given by
$$
\begin{pmatrix}
1 & 0 & \dots & 0 & a_1 & 0 \\
0 & 1 & \dots & 0 & a_2 & 0 \\
\vdots & \vdots & \ddots & \vdots & \vdots & \vdots \\
a_2 & a_1 & \dots & 1 & a_{k-2} & 0 \\
0 & 0 & \dots & 0 & 1 & 0 \\
0 & 0 & \dots & 0 & 0 & 1  
\end{pmatrix},
$$
where $a_1, \dots, a_{k-2} \in \F_q$ are isometries of $\beta$. Similarly, if $k$ is even, the following linear maps are isometries of $\beta$ 
\[
\begin{pmatrix}
	A & \begin{matrix} 0 & 0 \\ \vdots & \vdots\end{matrix} \\
	\begin{matrix} 0 & 0 & \dots \\ 0 & 0 & \dots \end{matrix} & \begin{matrix} 1 & 0 \\ 0 & 1 \end{matrix}
\end{pmatrix},
\]
where $A$ is an isometry of the alternating bilinear form $x_1 y_2 + x_2 y_1 + \dots + x_{k-3} y_{k-2} + x_{k-2} y_{k-3}$. Note that $\perp$ induces a pseudo-symplectic polarity on $H_{t}$, say $\perp_{|H_{t}}$, whose absolute points are those of $H_t \cap H_{\infty}$ and $H_{t} \cap H_{\infty} = (\ell \cap H_{t})^{\perp_{|H_{t}}}$. The projectivities preserving $\perp_{|H_{t}}$ are associated with the isometries given above and the group of these projectivities is transitive on $H_{t} \setminus (\ell \cup H_\infty)$.

Hence, in both cases, the group of projectivities associated with these maps is transitive on points of $H_t \setminus (\ell \cup H_\infty)$, $t \in T$. 
\end{proof}

\begin{prop}
	If $k \ge 3$, the graph $\cH(k,q)$ is an $(n,d,\lambda)$-graph where $\lambda = O(\sqrt{d})$ and 
	\[(n,d) = \begin{dcases}
		\left(\frac{q^{k-1}}{2},\frac{q^{k-2}}{2}\right) & \mbox{ if } k \mbox{ is odd, } \\
		\left(\frac{q^{k-1}-q}{2},\frac{q^{k-2}}{2}\right) & \mbox{ if } k \mbox{ is even.}
	\end{dcases}\]
\end{prop}
\begin{proof}
Since $|H_t \setminus H_\infty| = q^{k-2}$ and $H_t \cap H_{t'} = H_t \cap H_\infty = H_{t'} \cap H_\infty$, for $t, t' \in T$, $t \ne t'$, it follows that $|\cup_{t \in T} H_t \setminus H_\infty| = q^{k-1}/2$. Recall that when $k$ is odd, the line $\ell$ is contained in $H_0: X_{k-1} = 0$ which has no impact on $V(\cH(k,q))$, as $0 \notin T$. On the other hand for $k$ even, one point is removed from every $H_t, t \in T$.

Now consider a point $\x \in V(\cH(k,q))$; hence there exists $t_1 \in T$ such that $\x \in H_{t_1}$. By Lemma \ref{transitivity0}, we may assume without loss of generality that $\x$ equals $(0, \dots, 0, t_1, t_0)$, if $k$ is odd, or $(0, \dots, 0, 1, t_1, t_0)$, if $k$ is even. In the former case $\x^\perp : t_1 X_{k-2} + t_0 X_k = 0$ and $\x^\perp \cap \left(H_t \setminus \left(H_\infty \cup \ell\right)\right)$, $t \in T$, consists of the $q^{k-3}$ points $(\ast, \dots, \ast, t_0^2, t t_1, t_0 t_1)$. In the latter case  $\x^\perp : X_{k-3} + t_0 X_{k-1} + (t_0 + t_1) X_k = 0$ and $\x^\perp \cap \left(H_t \setminus \left(H_\infty \cup \ell\right)\right)$, $t \in T$, consists of the $q^{k-3}$ points $(\ast, \dots, \ast, t_0(t + t_0 + t_1), \ast, t, t_0)$. Note that $t \ne t_0 + t_1$, otherwise $1 = \Tr(t) = \Tr(t_0 + t_1) = \Tr(t_0) + \Tr(t_1) = 0$, a contradiction.

The fact that $\lambda = O(\sqrt{d})$ follows from Proposition \ref{subgraph}.
\end{proof}

\begin{prop}
The graph $\cH(k,q)$ is $K_k$-free. 
\end{prop}
\begin{proof}
By induction on $k$. The statement holds true for $k = 2$. Indeed, in this case no two points of $\cup_{t \in T} H_t \setminus \left(H_\infty \cup \ell\right)$ are orthogonal with respect to $\beta$. To see this fact, let $\x = (t_1, t_0), \x' = (t_2, t_0) \in V(\cH(2,q))$ and observe that $\beta(\x, \x') = t_0 (t_0 + t_1 + t_2) \ne 0$, since $t_0, t_1, t_2 \in T$. 

Assume that $\cH(k,q)$ is $K_{k}$-free. We show that the claim holds true for $k+1$. Let us consider a point $\x \in V(\cH(k+1,q))$. Hence there is $t_1 \in T$ such that $\x \in \left(H_{t_1} \setminus \left(H_\infty \cup \ell\right)\right)$. By Lemma~\ref{transitivity0}, we may assume that $\x$ is the point $(0, 0, \dots, 0, t_1, t_0)$ or $(0, 0, \dots, 1, t_1, t_0)$, according as $k$ is even or odd respectively. Thus $\x^\perp \cap H_t$ has equations 

\begin{align}
	&	\begin{dcases} 
		t_1 X_{k-1} + t_0 X_{k+1} = 0 &  \\ 
		t_0 X_{k} + t X_{k + 1} = 0 & \label{hyp_even} \\ 
		\end{dcases} \\
	& \nonumber \\
	&	\begin{dcases} 
		X_{k-2} + t_0 X_{k} + (t_0 + t_1) X_{k+1} = 0 &  \\ 
		t_0 X_{k} + t X_{k + 1} = 0 & \label{hyp_odd}\\ 
		\end{dcases}
\end{align}
according as $k$ is even or odd respectively. The polarity induced by $\perp$ on $\x^\perp$ has bilinear form 
\begin{align}
& x_1 y_2 + x_2 y_1 + \dots + x_{k-1} y_{k} + x_{k} y_{k-1} + \left( \frac{t_1}{t_0} \right)^2 x_{k-1} y_{k-1} & \mbox{ if } k \mbox{ is even, } \label{form1_even}\\ 
& x_1 y_2 + x_2 y_1 + \dots + x_{k - 4} y_{k - 3} + x_{k - 3} y_{k - 4} + t_0 (x_{k} y_{k-1} + x_{k-1} y_{k}) \nonumber \\ 
& \quad + (t_0 + t_1) (x_{k+1} y_{k-1} + x_{k-1} y_{k+1}) + x_{k} y_{k+1} + x_{k+1} y_{k} + x_{k+1} y_{k+1} & \mbox{ if } k \mbox{ is odd.}  \label{form1_odd}
\end{align}

By applying the invertible linear map 

\begin{table}[h]
	\centering
	\begin{tabular}{l l}
		$\begin{cases}
			X'_i = X_i, \;\; 1 \le i \le k-2, \\
			X_{k-1}' = \frac{t_0}{t_1} X_k, \\
			X_k' = \frac{t_1}{t_0} X_{k-1}, \\
			X_{k+1}' = X_{k+1},
		\end{cases}$
	& $\begin{cases}
		X'_i = X_i, \;\; 1 \le i \le k-3, \\
		X_{k-2}' = t_0 X_{k-1} + X_{k+1}, \\
		X_{k-1}' = X_k + \frac{t_0 + t_1}{t_0} X_{k+1}, \\
		X_{k}' = X_{k+1}, \\
		X_{k+1}' = X_{k-2},
	\end{cases}$ \\
    $\mbox{}$ \\
	if $k$ is even and & if $k$ is odd,
	\end{tabular}
\end{table}	
it is easily seen that the bilinear forms \eqref{form1_even}, \eqref{form1_odd} are equivalent to the bilinear forms of $\F_q^k$ given by
\begin{align*}
& x'_1 y'_2 + x'_2 y'_1 + \dots + x'_{k-1} y'_{k} + x'_{k} y'_{k-1} + x'_{k} y'_{k} & & \mbox{ if } k \mbox{ is even, } \\ 
& x'_1 y'_2 + x'_2 y'_1 + \dots + x'_{k-1} y'_{k-2} + x'_{k-2} y'_{k-1} + x'_{k} y'_{k} & & \mbox{ if } k \mbox{ is odd, } 
\end{align*}
and the hyperplanes $\x^\perp \cap H_t$ \eqref{hyp_even}, \eqref{hyp_odd} of $\x^\perp$ are mapped to the hyperplanes of $\PG(k-1,q)$ given by
\begin{align*}
& t_1 X'_{k-1} + t X'_k = 0 & & \mbox{ if } k \mbox{ is even, } \\ 
& t_0 X'_{k-1} + (t_0 + t_1 + t) X'_k = 0 & & \mbox{ if } k \mbox{ is odd. } 
\end{align*}
As $t_1 \in T$ and $\{t \,\, | \,\, t \in T\} = \{t_0 + t_1 + t \,\, | \,\, t \in T\}$ respectively, the assertion follows by the induction hypothesis.
\end{proof}

\section{Conclusion}

In this short note, we gave a new construction for optimally pseudorandom $K_k$-free graphs of edge density $d/n = \Theta(n^{-1/(k-1)})$ from polarity graphs, where the characteristic of the underlying finite field is even. It would be interesting to provide geometrical constructions leading to even denser pseudorandom graphs, ideally breaking the barrier of $\Theta(n^{-1/k})$. The idea of taking large subgraphs of polarity graphs applied here and in \cite{BIP2020} seems to reach its limits, at least for case of $k=3$ on which the induction relies: a $K_3$-free graph obtained from any polarity graph $\Gqthree$ has at most $\alpha(\Gqthree) = \Theta(q^{3/2})$ vertices as shown by Mubayi and Williford \cite{MW07}. 

\smallskip
{\footnotesize
\noindent\textit{Acknowledgments.}
This work was supported by the Italian National Group for Algebraic and Geometric Structures and their Applications (GNSAGA-- INdAM).
}

\bibliographystyle{alpha}

\end{document}